\documentclass[dvipsnames,letterpaper, 10 pt, conference]{ieeeconf}

\IEEEoverridecommandlockouts	

\usepackage{cite}
\usepackage{amsmath,amssymb,amsfonts}
\usepackage{algorithmic}
\usepackage{textcomp}
\usepackage{graphicx,color}
\usepackage{mathrsfs}
\usepackage[vlined,ruled]{algorithm2e}
\usepackage{subfigure}
\usepackage{url}
\usepackage{color}
\usepackage{dsfont}
\usepackage{bbm}
\usepackage{booktabs}
\usepackage{array}
\usepackage[table]{xcolor} 
\usepackage{yfonts}
\usepackage[normalem]{ulem}
\usepackage{arydshln}
\usepackage{multicol}
\usepackage{amsmath}
\usepackage{stfloats}

  \newenvironment{smallarray}[1]
 {\null\,\vcenter\bgroup\scriptsize
  \arraycolsep=.13885em
  \hbox\bgroup$\array{@{}#1@{}}}
 {\endarray$\egroup\egroup\,\null}

\newtheorem{theorem}{Theorem}[section]

\newtheorem{remark}{Remark}
\newtheorem{assumption}[theorem]{Assumption}
\newtheorem{example}{Example}
\newtheorem{appxlem}{Lemma}[subsection]

\newtheorem{appxthrm}[appxlem]{Theorem}

\newcommand{\map}[3]{#1: #2 \rightarrow #3}

\newcommand{\subscr}[2]{{#1}_{\textup{#2}}}

\newcommand{\until}[1]{\{1,\dots,#1\}}

\newcommand{\blkdiag}{\mathrm{blkdiag}}
\newcommand{\Ubar}{\overbar{U}}

\usepackage{tabularx}
\usepackage{multirow}
\usepackage{makecell}

\newcommand\aamsout{\bgroup\markoverwith{\textcolor{red}{\rule[0.5ex]{2pt}{1pt}}}\ULon}

\newcommand{\Rank}{\operatorname{Rank}}

\newcommand{\real}{\mathbb{R}}

\newcommand{\transpose}{\mathsf{T}} 
\newcommand{\T}{\mathsf{T}} 

\newcommand{\mc}{\mathcal}

\newcommand{\lqr}{\text{LQR}}
\newcommand{\kf}{\text{KF}} 

\newcommand{\expect}[1]{\mathbb{E}\left[#1\right]}
\newcommand{\Tr}[1]{\mathrm{tr}\left[#1\right]}

\newcommand{\overbar}[1]{\mkern 1.5mu\overline{\mkern-1.5mu#1\mkern-1.5mu}\mkern 1.5mu}
\newcommand{\Vect}{\text{vec}}

\newcommand{\1}{\mathds{1} }
\newcommand{\D}{\text{D}}

\newcommand{\dlqg}{\text{dLQG}}
\newcommand{\lqg}{\text{LQG}}
\newcommand{\dd}{\text{d}}

\newcommand{\argmin}[2] {\mathrm{arg}\min_{#1}#2}

\DeclareSymbolFont{bbold}{U}{bbold}{m}{n}
\DeclareSymbolFontAlphabet{\mathbbold}{bbold}

\newcommand\oprocendsymbol{\hbox{$\square$}}
\newcommand\oprocend{\relax\ifmmode\else\unskip\hfill\fi\oprocendsymbol}

\renewcommand{\theenumi}{(\roman{enumi})}

\newcommand*{\QEDA}{\hfill\ensuremath{\blacksquare}}%

\graphicspath{{figs/}}

\makeatletter
\let\NAT@parse\undefined
\makeatother
\usepackage[colorlinks,urlcolor=blue, linkcolor=blue, citecolor=blue]{hyperref}

\usepackage{mathtools}

\begin{document}

\title{\LARGE \bf On the Sample Complexity of the Linear Quadratic
  Gaussian Regulator}


\author{Abed~AlRahman~Al~Makdah~and~Fabio~Pasqualetti \thanks{This
    paper is based upon work supported in part by awards
    ONR-N00014-19-1-2264, AFOSR-FA9550-20-1-0140, and
    AFOSR-FA9550-19-1-0235.  A. A. Al Makdah and F. Pasqualetti are
    with the Department of Electrical and Computer Engineering and
    Mechanical Engineering at the University of California, Riverside,
    \href{mailto:aalm005@ucr.edu}{\{\texttt{aalm005}},\href{mailto:fabiopas@engr.ucr.edu}{\texttt{fabiopas\}@ucr.edu}}.}}

\maketitle

\begin{abstract}
  In this paper we provide direct data-driven expressions for the
  Linear Quadratic Regulator (LQR), the Kalman filter, and the Linear
  Quadratic Gaussian (LQG) controller using a finite dataset of noisy
  input, state, and output trajectories. We show that our data-driven
  expressions are consistent, since they converge as the number of
  experimental trajectories increases, we characterize their
  convergence rate, and we quantify their error as a function of the
  system and data properties. These results complement the body of
  literature on data-driven control and finite-sample analysis, and
  they provide new ways to solve canonical control and estimation
  problems that do not assume, nor require the estimation of, a model
  of the system and noise and do not rely on solving implicit
  equations.
\end{abstract}

\section{Introduction}\label{sec:
  introduction} Consider the discrete-time, linear, time-invariant
system
\begin{equation}\label{eq: system}
  \begin{aligned}
    x(t+1) &= A x(t) +B u(t)+ w(t),\\
    y(t)  & = Cx(t)+v(t), \qquad t\geq 0,
  \end{aligned}
\end{equation}
where $x(t)\in\real^{n}$ denotes the state, $u(t)\in \real^{m}$ the
control input, $y(t)\in\real^{p}$ the measured output, $w(t)$ and
$v(t)$ the process and measurement noise at time $t$. The LQG control
problem asks for the input that~minimizes the cost function
\begin{align}\label{eq: LQG cost}
  \lim_{T\rightarrow \infty}\mathbb{E} \left[
  \frac{1}{T}\Big(\sum_{t=0}^{T-1}x(t)^{\T}Q_x x(t) +
  u(t)^{\T} R_u u(t)\Big) \right] , 
\end{align}
where $Q_x\succeq 0$, $R_u\succ 0$ are weight matrices and $T$ is the
control horizon. With the standard assumptions that\footnote{These
  assumptions also hold throughout this paper.}
\begin{itemize}
\item[(A1)] the process and measurement noise sequences and the
  initial state are independent at all times and satisfy
  $w(t)\!\sim\!\mc{N}(0,Q_w)$, $v(t)\!\sim\!\mc{N}(0,R_v)$, and
  $x(0)\!\sim\!\mc{N}(0,\Sigma_0)$, with $Q_w\succeq 0$, $R_v\succ 0$,
  and $\Sigma_0\succ 0$;

\item[(A2)] the pairs $(A, B)$ and $(A, Q_w^{\frac{1}{2}})$ are
  controllable, and the pairs $(A,C)$ and $(A,Q_x^{\frac{1}{2}})$ are
  observable;
\end{itemize}
the input that solves the LQG problem can obtained by concatenating
the Kalman filter for \eqref{eq: system} with the (static) controller
that solves the LQR problem for \eqref{eq: system} with weight
matrices $Q_x$ and $R_u$ \cite{KZ-JCD-KG:96}. That is,
\begin{align}\label{eq: LQR KF}
  u^*(t) = \subscr{K}{LQR} \subscr{x}{KF} (t),
\end{align}
where $\subscr{x}{KF} (t)$ is the Kalman estimate of the state
$x(t)$. The classic, model-based computation of the LQR gain and
Kalman filter in \eqref{eq: LQR KF} requires the complete knowledge of
the system \eqref{eq: system}, including the noise
statistics. Motivated by the recent successes of data-driven and
machine-learning methods, we seek here a solution to the LQG problem
that relies only on a (finite) dataset of experimental data, without
the need to estimate the system dynamics and noise statistics.

\noindent
\textbf{Related work.} Data-driven methods for system analysis and
control have flourished in the last years and are revolutionizing the
field \cite{BR:18}. The methods developed in this paper fall in the
category of direct data-driven methods \cite{GB-DSB-FP:20}, where
controls are obtained directly from data bypassing the classic system
identification step \cite{MG:05}. In line with earlier work and
differently from optimization-based approaches
\cite{BH-KZ-NL-MM-MF-TB:23,FD-PT-CDP:22}, we pursue here closed-form
data-driven expressions, which are typically computationally
advantageous \cite{FC-GB-FP:22}, are transparent, and can reveal novel
insights into the~problems~\cite{FC-FP:22}.

This paper focuses on data-driven LQG control, while most of the
literature on data-driven control has focused on the LQR problem with
noiseless data \cite{IM-PR:07,GRGS-ASB-CL-LC:19,MR-CDP-PT:20}. Recent
work \cite{CYC-AB:22} has studied the design of data-driven
controllers from noisy data \cite{CDP-PT:21,BG-PME-TS:20}, the design
of data-driven Kalman filters \cite{XZ-BH-TB:23}, imitation-based LQG
control design \cite{TG-AAAM-VK-FP:23}, and some versions of the
output-weighted LQG control problem \cite{MF-BDM-PVO-MG:99, RES-GS:94}. Compared
to \cite{RES-GS:94}, in particular, this paper does not assume perfect
knowledge of the Markov parameters or any part of the system dynamics
and noise, and it does not estimate them to solve the state-weighted
LQG problem. To the best of our knowledge, this paper contains the
first direct, closed-form data-driven solution to the state-weighted
LQG problem, with finite-sample performance guarantees.

The recent literature on the analysis of the sample complexity of
estimation and control problems is also relevant to this work. In
particular, \cite{SD-HM-NM-BR-ST:20,YZ-LF-MK-NL:21} follow an indirect
approach, where sample complexity bounds are derived for the
identification of the system dynamics and such errors are propagated
towards the design of LQR and LQG controllers. Differently from this
paper, this analysis is valid only for stable systems and
output-weighted LQG costs. Bounds on the performance of the learned
LQG controller are also derived in \cite{HM-ST-BR:19} assuming a
sufficiently small error in the system identification step
\cite{SO-NO:22,AT-GJP:19,YZ-NL:21,ST-RF-MS:23}, and in
\cite{HM-AZ-MS-MRJ:21} where the optimal LQR is learned in a
model-free setting using gradient methods. Although this paper makes
use of similar technical tools, the approach pursued here is direct
and does not rely on the identification of the system matrices, nor on
optimization algorithms to design or tune robust controllers. Further,
this paper considers the canonical LQG setting, rather than the noisy
LQR problem or the output-weighted LQG problem with noisy controls,
and it provides closed-form expressions for the optimal controllers
rather than their performance.

\noindent
\textbf{Contributions of the paper.} The main contributions of this
paper is the characterization of direct data-driven formulas for the
LQR gain, Kalman filter, and LQG gain using a dataset of trajectories
of the input, state, and output of the system \eqref{eq:
  system}. Importantly, since the experimental data is noisy and the
system dynamics and noise statistics are unknown, we show that our
formulas are consistent, as they converge to the true expressions when
the amount of experimental data increases. Additionally, we
characterize the convergence rate of our expressions, as well as their
error when the data is of finite size. Finally, we provide
illustrative examples and remarks to highlight how the properties of
the system and of the experimental data affect the accuracy of our
formulas.

\noindent
\textbf{Organization of the paper.} The remainder of the paper is
organized as follows. Section \ref{sec: setup} formalizes our problem
setting and contains some preliminary results. Section \ref{sec: main
  results} contains our main results and examples, and Section
\ref{sec: conclusion} concludes the paper. Finally, all proofs are in
the Appendix.

\noindent
\textbf{Notation.} A Gaussian random variable $x$ with mean $\mu$ and
covariance $\Sigma$ is denoted as $x\sim\mc{N}(\mu,\Sigma)$. The
$n\times n$ identity matrix is denoted by $I_n$, and the $n \times m$
zero matrix is denoted by $0_{n\times m}$. The expectation operator is
denoted by $\mathbb{E}[\cdot]$. The trace of a square matrix $A$ is
denoted by $\Tr{A}$. A positive definite (semidefinite) matrix $A$ is
denoted as $A\succ 0$ ($A\succeq 0$). The Kronecker product is denoted
by $\otimes$, and the vectorization operator is denoted by
vec($\cdot$). The left (right) pseudo inverse of a tall (fat) matrix
$A$ is denoted by $A^{\dagger}$. A block-diagonal matrix with block
matrices $A$ and $B$ is denoted by $\blkdiag(A,B)$. The smallest
(largest) singular value of a matrix $A$ is denoted by
$\sigma_{\text{min}}(A)$ \big($\sigma_{\text{max}}(A)$\big).

\section{Problem formulation and preliminary results}\label{sec:
  setup}
In this work we aim to compute the LQG inputs in a data-driven setting
where datasets from offline experiments are available but the system
matrices and noise statistics are unknown. In particular, we have
access to the following
data:
\begin{align}\label{eq: data}
  U\! =\!
  \begin{bmatrix}
    u^{1} \!\!\!&\!\! \cdots \!\!&\!\!\! u^{N}
  \end{bmatrix}\!,~
                                   X \!=\!
                                   \begin{bmatrix}
                                     x^{1} \!\!\!&\!\! \cdots \!\!&\!\!\! x^{N}
                                   \end{bmatrix}\!,~
                                                                    Y\!=\!
                                                                    \begin{bmatrix}
                                                                      y^{1} \!\!\!&\!\! \cdots \!\!&\!\!\! y^{N}
                                                                    \end{bmatrix}\!,
\end{align}
where $x^i$ and $y^i$ are the $i$-th state and output trajectories of
\eqref{eq: system} generated by the input $u^i$. That is, for
$i \in \until{N}$,
\begin{align*}
  u^{i}& =
         \begin{bmatrix}
           u^{i}(0) \\ \vdots \\ u^{i}(T-1)
         \end{bmatrix}
  ,
    x^{i} =
         \begin{bmatrix}
           x^{i}(0) \\ \vdots \\ x^{i}(T)
         \end{bmatrix}
  ,
    y^{i} =
         \begin{bmatrix}
           y^{i}(0) \\ \vdots \\ y^{i}(T)
         \end{bmatrix}
  ,
\end{align*}
where $T$ is the horizon of the control experiments. We make the
following assumption on the experimental inputs.

\begin{assumption}{\bf \emph{(Experimental inputs)}}\label{assump:
    data input}
  The inputs in \eqref{eq: data} are independent and identically
  distributed, that is, $u^i(t) \sim \mathcal{N}(0,\Sigma_u)$, with
  $\Sigma_u\succ 0$, for all $i \in \until{N}$ and times.\oprocend
\end{assumption}

In our analysis we will make use of an equivalent characterization of
the LQG inputs derived in \cite[Theorem 2.1]{AAAM-VK-VK-FP:22}, which
shows that these inputs can also be computed~as
\begin{align}\label{eq: static LQG}
  u^*(t+n) = \subscr{K}{LQG}
  \begin{bmatrix}
    u^*(t) \\ \vdots \\ u^*(t + n - 1) \\ y^*(t + 1) \\ \vdots \\ y^*(t+n)
  \end{bmatrix},
\end{align}
where the static gain $\subscr{K}{LQG}$ depends on the system and
noise matrices, and $y^*$ is the output of \eqref{eq: system} with
input $u^*$.

\begin{remark}{\bf \emph{(State vs output measurements)}}\label{rmrk:
    data} We assume here that the
  state of the system \eqref{eq: system} can be directly
  measured. This assumption is easily satisfied in certain lab
  experiments, where additional sensors (e.g., a motion capture system
  for robotic applications) can be deployed during the design stage to
  measure the system state and collect training data. Further, state
  measurements are necessary to solve the state-weighted LQG problem,
  since the state weight matrix $Q_x$ uses specific coordinates that
  cannot be inferred from output measurements only
  \cite{AT-IZ-NM-GJP:22}, but they can be substituted with input and
  output measurements for different versions of the LQG problem. See
  also \cite{AAAM-VK-VK-FP:22} for a reformulation of the LQG problem
  that uses only input and output measurements.  \oprocend
\end{remark}

\section{Data-driven formulas for LQG control}\label{sec: main
  results}
In this section we derive our main results, that is, direct
data-driven formulas for the LQR controller, the Kalman filter, and
the LQG controller using the data \eqref{eq: data}. Additionally, we
show that these formulas are consistent, i.e., they converge to the
true model-based expressions as the data grows, and we finally
quantify their error when the data is finite.

We start by introducing some additional notation. Let
\begin{align}\label{eq: X0}
  X_t =
  \begin{bmatrix}
    x^1(t)^\transpose & \cdots & x^N (t)^\transpose
  \end{bmatrix}^\transpose,
\end{align}
and, given input and state trajectories $\subscr{u}{v} \in \real^{mT}$
and $\subscr{x}{v} \in \real^{nT}$, let
$\subscr{u}{m} \in \real^{m \times T}$ and
$\subscr{x}{m} \in \real^{n \times T}$ be the matrices obtained by
reorganizing the inputs and states in the vectors $\subscr{u}{v}$ and
$\subscr{x}{v}$ in chronological order. The next result characterizes
the LQR gain from data.


\begin{theorem}{\bf \emph{(Data-driven LQR gain)}}\label{thm: LQR
    gain}
  Let $x_0 \in \real^n$ and
  \begin{align}\label{eq: LQR trajectories}
    \begin{bmatrix}
      \subscr{u}{v} \\ \subscr{x}{v}
    \end{bmatrix}
    =
    \begin{bmatrix}
      H \\ M
    \end{bmatrix}
    P^{-1/2}\left(
    \begin{bmatrix}
      I_n \!\!&\!\! 0_{n\times mT}
    \end{bmatrix}
            P^{-1/2}\right)^{\dagger} x_0 ,
  \end{align}
  where
  \begin{align}\label{eq: M and P} 
    \begin{split}
      H &= 
      \begin{bmatrix}
        0_{mT \times n} \!\!&\!\! I_{mT}
      \end{bmatrix}, 
      M = X
      \begin{bmatrix}
        X_0 \\ U
      \end{bmatrix}^\dag
      , \text{ and }\\
      P &= M^\transpose \left( I_{T+1}\!\otimes\! Q_x \right) M +
      \blkdiag\left(0_{n\times n }, I_{T}\! \otimes\! R_u\right) .
    \end{split}
  \end{align}
  Let $\subscr{x}{v}^* \in \real^{nT}$ be the trajectory of \eqref{eq:
    system} with initial state $x_0$, control input
  $u^*(t) = \subscr{K}{LQR} x(t)$, and $w(t) = 0$ at all times. Then,
  the data-driven estimate
  $K_\text{LQR}^\text{D} = \subscr{u}{m} \subscr{x}{m}^{\dag}$ of
  $K_\text{LQR}$
  \begin{align*}
    \| \subscr{K}{LQR} - K_\text{LQR}^\text{D} &\|_2 \le
                                                 \frac{1}{\subscr{\sigma}{min}(\subscr{x}{m}^*)\left(1-
                                                 \kappa(\subscr{x}{m}^*)\right)}
                                                 \left(
                                                 \frac{c_1}{\sqrt{N}}
                                                 + c_2 \rho^T \right)\!, 
  \end{align*}
  for sufficiently large $N$ and probability at least $1-6\delta$,
  where
  \begin{align*}
    \kappa (\subscr{x}{m}^*) =
    \frac{\subscr{\sigma}{max} (\subscr{x}{m} - \subscr{x}{m}^*)}{
    \subscr{\sigma}{min} (\subscr{x}{m}^*)},
  \end{align*}
  where the constants $c_1$ and $c_2$ are independent of $N$ and
  are defined in \eqref{eq: c1 and c2},
  $\rho<1$, and $\delta \in [0,1/6]$. \oprocend
\end{theorem}

\smallskip A proof of Theorem \ref{thm: LQR gain} is postponed to
Appendix \ref{app: LQR}. Some comments are in order. First, Theorem
\ref{thm: LQR gain} provides a direct, data-driven way to estimate the
LQR gain from noisy data, namely, $K_\text{LQR}^\text{D}$, and
characterizes the error between the true and the estimated gains. Such
error vanishes as the number ($N$) and length ($T$) of the
experimental trajectories grow.\footnote{The constant $c_1$, as well
  as other constants defined later in the paper, depend also on the
  horizon $T$. While a detailed characterization of the effects of
  this dependency requires a dedicated analysis, notice that our
  expressions remain consistent if $N$ grows sufficiently faster than
  $T$. The formulas in the paper quantify the error for finite choices
  of these two parameters.}  Further, the term
$\kappa (\subscr{x}{m}^*)$ also vanishes as the number of experimental
trajectories increases (see Theorem \ref{thrm: bound on LQR
  traj}). Second, the
vectors $\subscr{u}{v}$ and $\subscr{x}{v}$ contain an estimate of the
optimal input and state trajectories that minimize the LQR cost with
matrices $Q_x$ and $R_u$ for the system \eqref{eq: system} with
initial state $x_0$ and without process noise.  Notably, these
trajectories are estimated using the noisy dataset \eqref{eq:
  data}. Thus, this result extends the analysis in
\cite{FC-GB-FP:22}. Third, Theorem \ref{thm: LQR gain} is valid when
$N$ is sufficiently large. In particular, $N$ needs to be at least
large enough to satisfy $\kappa (\subscr{x}{m}^*) < 1$ (see Appendix
\ref{app: LQR} for other conditions on $N$). Also, the result holds
with probability $1-6\delta$, and the specific choice of $\delta$
affects the magnitude of the constant $c_1$. Fourth and finally,
although formulas with similar convergence rates for the estimation of
the LQR exist \cite{SD-HM-NM-BR-ST:20,HM-ST-BR:19}, Theorem \ref{thm:
  LQR gain} provides an alternative, direct, closed-form expression of
the gain, as opposed to indirect and optimization-based
approaches. This will allow us to estimate the LQG controller.
%
\begin{figure*}[!t]
  \centering
  \includegraphics[width=1\textwidth,trim={0cm 0cm 0cm
    0cm},clip]{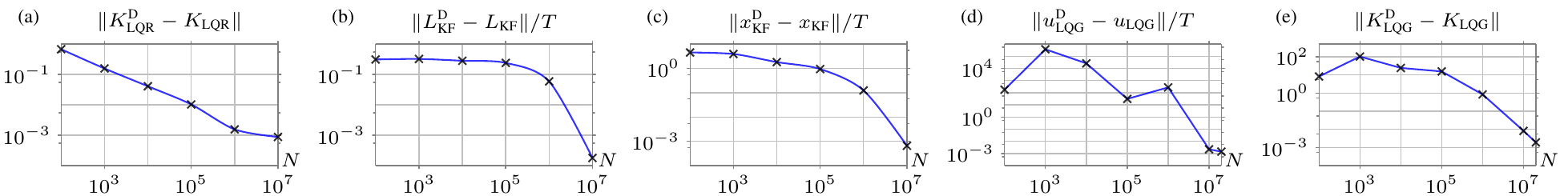}
  \caption{This figure shows the error between the data-driven and the model-based gains as a function of the size of the data in \eqref{eq: data} for the setting described in Example \ref{example: LQR}, \ref{example: KF}, and \ref{example: LQG}. Panel (a) shows the error between the model-based LQR gain and the data-based LQR gain obtained from Theorem~\ref{thm: LQR gain} as a function of $N$ for the setting described in Example \ref{example: LQR}. Panel (b) shows the error between the model-based Kalman filter and the data-based Kalman filter obtained from Theorem~\ref{thm: KF}, and panel (c) shows the error between the corresponding state estimates as a function of $N$ for the setting in Example \ref{example: KF}. Panel (d) shows the error between the model-based LQG input generated by \eqref{eq: LQR KF} and the data-based LQG input generated by \eqref{eq: LQG inputs} as a function of $N$. Panel (e) shows the error between the model-based LQG gain in \eqref{eq: static LQG} and the data-based LQG gain obtained from Theorem~\ref{thm: LQG} as a function of $N$ for the setting in Example \ref{example: LQG}. We observe that all the quantities in the plots decrease as the number of trajectories, $N$, increases, which agrees with our theoretical results.
}
    \label{fig: example_1}
\end{figure*}
\begin{example}{\bf \emph{(Estimating the LQR gain from noisy
      data)}}\label{example: LQR}
  Consider System \eqref{eq: system} with
\begin{align*}
A=
\begin{bmatrix}
 0.7 & 1.2 \\ 0 & 0.4
\end{bmatrix},\quad
B=
\begin{bmatrix}
 0\\1
\end{bmatrix}, \quad
C=
\begin{bmatrix}
 1 & 0
\end{bmatrix},
\end{align*}
$Q_x\!\!=\!\!5I_2$, $Q_w\!\!=\!\!2I_2$, $R_u\!=\!R_v\!=\!\Sigma_u\!\!=\!\!1$, and $\Sigma_0\!=\!I_2$. We
collect open-loop trajectories as in \eqref{eq: data} generated by
inputs satisfying Assumption \ref{assump: data input} with horizon
$T\!\!=\!\!50$. The model-based LQR gain is $K_\lqr\!\!=\!\![0.241\;\; 0.788]$. We use
Theorem~\ref{thm: LQR gain} to compute the data-driven LQR gain,
$K_\lqr^\D$ for different values of $N$. Fig.~\ref{fig: example_1}(a)
shows the error $\|K_\lqr^\D\!\!-\!K_\lqr\|$ as a function of the number of trajectories.~ \oprocend

\end{example}

\smallskip

We now focus on estimating the Kalman filter from noisy data with
unknown system dynamics and noise statistics.

\begin{theorem}{\bf \emph{(Data-driven Kalman filter)}}\label{thm: KF}
  Let $U_t$ and $Y_t$ be the submatrices of $U$ and $Y$ in \eqref{eq:
    data} obtained by selecting only the inputs and outputs up to time
  $t$. Let
  \begin{align}\label{eq: DD KF}
    L_t^{\D} = X_t
    \begin{bmatrix}
      U_{t-1} \\ Y_t
    \end{bmatrix}^\dag,
  \end{align}
  where $X_t$ is as in~\eqref{eq: X0}. Then, for every $t \in [0,T]$,
  \begin{align}\label{eq: DD KF error}
    \left\| \subscr{x}{KF} (t) - L_t^{\D}
    \begin{bmatrix*}[c]
      u_0^{t-1} \\ y_0^t
    \end{bmatrix*}
    \right\|_2
     \le \frac{c_3}{\sqrt{N}}{\left\|
    \begin{bmatrix*}
      u_0^{t-1} \\ y_0^t
    \end{bmatrix*}\right\|}_2,
  \end{align}
  with probability at least $1-2\delta$, where $u_0^{t}$ and $y_0^t$
  are the vectors of inputs and outputs of~\eqref{eq: system},
  respectively, from time $0$ up to time $t$, $c_3$ is a constant
  independent of $N$ as defined in \eqref{eq: state estimate bound}, and $\delta \in [0,1/2]$.~\oprocend
\end{theorem}


\smallskip A proof of Theorem \ref{thm: KF} is postponed to Appendix
\ref{app: KF}. Theorem \ref{thm: KF} provides a way to construct an
approximate Kalman filter using a finite set of experimental data,
without knowing the system dynamics and the statistics of the noise.
As can be seen from \eqref{eq: DD KF error}, the error vanishes with
rate $1/\sqrt{N}$ as the number of experimental data grows, showing
the consistency of the data-driven Kalman filter expressions~\eqref{eq: DD KF}.

\begin{example}{\bf \emph{(Estimating the Kalman filter from noisy
      data)}}\label{example: KF} Following the setting introduced in
  Example \ref{example: LQR}, we use Theorem \ref{thm: KF} to obtain
  the data-driven Kalman filter, $L_\kf^{\D}$, and the corresponding
  data-driven state estimate, $x_\kf^{\D}$. Fig. \ref{fig:
    example_1}(b) and \ref{fig: example_1}(c) show the errors
  $\|L_\kf^\D\!\!-\!L_\kf\|$ and $\|x_\kf^\D\!\!-\!x_\kf\|$.~
  \oprocend
\end{example}
\smallskip

Theorems \ref{thm: LQR gain} and \ref{thm: KF} allow us to compute the
LQG inputs from time $0$ up to time $T$. In particular, recalling the
structure of the LQG inputs due to the separation principle
\cite{KZ-JCD-KG:96},
\begin{align}\label{eq: LQG inputs}
  u_{\dlqg}(t) = K_\text{LQR}^\text{D} \underbrace{L_t^{\D}
  \begin{bmatrix}
    u_\dlqg(0) \\ \vdots \\ u_\dlqg(t - 1) \\ y_\dlqg(0) \\ \vdots \\ y_\dlqg(t)
  \end{bmatrix}}_{x_\text{KF}^\text{D} (t)}
  ,
\end{align}
where $x_\text{KF}^\text{D}$ is the state estimate obtained using our
data-driven scheme. Fig. \ref{fig: example_1}(d) shows how these
data-driven inputs compare to the model-based LQG inputs as a function
of the amount of data. As expected, the performance gap between the
data-driven and the model-based schemes shrinks as the amount of data
increases. We next provide an estimate of the LQG gain \eqref{eq:
  static LQG}, which allows us to compute LQG inputs beyond the
horizon $T$ of the experimental trajectories. We start by collecting $M\geq n+nm+np$ closed-loop input-output
trajectories of system \eqref{eq: system} driven by the LQG inputs  
generated from \eqref{eq: LQG inputs}. In particular,
\begin{align}\label{eq: cl data}
  {U}_\dlqg \!=\!\!
  \begin{bmatrix}
    u_{\dlqg}^{1} \!\!\!&\!\! \cdots \!\!&\!\!\! u_{\dlqg}^{M}
  \end{bmatrix}\!, ~
                                                                    Y_\dlqg\! =\!
                                                                    \begin{bmatrix}
                                                                      y_{\dlqg}^{1} \!\!\!&\!\! \cdots \!\!&\!\!\! y_{\dlqg}^{M}
                                                                    \end{bmatrix}\!,
\end{align}
where $y_{\dlqg}^i$ is the $i$-th output trajectory of \eqref{eq:
  system} generated by the LQG input $u_\dlqg^i$ in \eqref{eq: LQG
  inputs}. That is, for $i \in \until{M}$,
\begin{align*}
  u_\dlqg^{i}&\! =\!\!
         \begin{bmatrix}
           u_\dlqg^{i}(0) \\ \vdots \\ u_\dlqg^{i}(T\!-\!1)
         \end{bmatrix}\!\!, \quad
    y_\dlqg^{i}\! =\! \!
         \begin{bmatrix}
           y_\dlqg^{i}(0) \\ \vdots \\ y_\dlqg^{i}(T)
         \end{bmatrix}\! . 
\end{align*}

%


\begin{theorem}{\bf \emph{(Data-driven LQG gain)}}\label{thm: LQG}
  Let $U^n_{\dlqg}$ and $Y_{\dlqg}^n$ be the submatrices of $U_\dlqg$
  and $Y_\dlqg$ in \eqref{eq: cl data} obtained by selecting only the
  inputs from time $T-n$ up to time $T-1$ and the outputs from time
  $T-n+1$ up~to~time~$T$, respectively. Define the data-driven LQG
  gain as
  \begin{align*}
    K_\text{LQG}^\text{D} =&
                             K_\text{LQR}^\text{D} L_t^{\D}   
                             \underbrace{\begin{bmatrix}
                               U_\dlqg\\ Y_\dlqg
                             \end{bmatrix}
                                \begin{bmatrix}
                                  U^n_{\dlqg}\\
       Y^n_{\dlqg}
     \end{bmatrix}^\dag}_{c_4},
  \end{align*}
  Then, the data-driven estimate of the LQG gain satisfies
  \begin{align}\label{eq: LQG bound}
    \| K_\text{LQG} - K_\text{LQG}^\text{D} \|_2 \leq 
    \|c_4\|_2\left(\frac{c_5+c_6\rho^T}{\sqrt{N}}  + c_7\rho^T\right),
  \end{align}
  for sufficiently large $T$ and $N$ and probability at least
  $1-8\delta$, where the constants $c_5$, $c_6$, and $c_7$ are
  independent of $N$ and are defined in \eqref{eq: c5 c6 c7}, $\rho<1$, and $\delta \in [0,1/8]$.  ~\oprocend
\end{theorem}

We postpone the proof of Theorem \ref{thm: LQG} to Appendix \ref{app:
  LQG}. Theorem \eqref{thm: LQG} provides a direct data-driven
expression of the LQG gain that converges with polynomial rate as the
experimental data increases. To the best of our knowledge, this result
is the first of its kind, and it provides a new way to compute the LQG
controller using offline experimental data and a finite number of
online experiments, without knowing or identifying the system and
noise matrices. 
\smallskip
\begin{example}{\bf \emph{(Estimating the LQG gain from noisy
      data)}}\label{example: LQG} Following the setting introduced in
  Example \ref{example: LQR} and Example \ref{example: KF}, we use
  Theorem \ref{thm: LQG} to obtain the data-driven LQG gain,
  $K_\lqg^{\D}$. Fig. \ref{fig: example_1}(e) shows the error
  $\|K_\lqg^\D\!\!-\!K_\lqg\|$ as a function of the number of
  trajectories, $N$.~ \oprocend
\end{example}

\section{Conclusion}\label{sec: conclusion}
In this paper we derive direct data-driven expressions for the LQR
gain, Kalman filter, and LQG controller using a dataset of input,
state, output trajectories. We show the convergence of these
expressions as the size of the dataset increases, we characterize
their convergence rate, and we quantify the error incurred when using
a dataset of finite size. Our expressions are direct, as they do not
use a model of the system nor require the estimation of a model, and
provide new insights into the solution of canonical control and
estimation problems. Directions of future research include the direct
data-driven solution to $\mc H_2$ and $\mc H_\infty$ problems, as well
as the extension of the results to accomodate for incomplete,
heterogeneous and, possibly, corrupted datasets.

\section*{Appendix}
\renewcommand{\thesubsection}{A}
\subsection{Technical lemmas}
\begin{appxlem}{\bf \emph{(Product of Gaussian matrices \cite[Lemma 1]{SD-HM-NM-BR-ST:20})}}\label{lemma: product of Gaussian matrices}
Let $A\!=\![a_1,\cdots\!, a_N]$ and $B\!=\![b_1,\cdots\!, b_N]$, where $\left. a_i\in \mathbb{R}^n\right.$ and $b_i \in \mathbb{R}^m$ are independent random vectors with $a_i \sim \mathcal{N}(0,\Sigma_a)$ and $b_i \sim \mathcal{N}(0,\Sigma_b)$ for $i=1,\cdots,N$. Let $\delta \in [0,1]$ and $N\geq 2(n+m) \log{(1/\delta)}$. Then, with probability at least $1-\delta$
\begin{align*}
{\| AB^{\T}\|}_2\leq 4 {\|\Sigma_a\|}_2^{1/2}{\|\Sigma_b\|}_2^{1/2} \sqrt{N(n+m)\log{(9/\delta)}}.
\end{align*}~\oprocend
\end{appxlem}
\begin{appxlem}{\bf \emph{(Singular values of a Gaussian matrix)}}\label{lemma: svd of Gaussian matrix} 
  Let $\delta \in [0,1]$, and let $A\in \mathbb{R}^{n\times N}$ be a
  random matrix with independent entries distributed as 
  $\mathcal{N}(0,1)$. Then, for
  $N\geq 8n+16\log{(1/\delta)}$, each of the following inequalities hold with probability probability at least $1-\delta$
\begin{align*}
 \subscr{\sigma}{min}(A) \geq \sqrt{N}/2, \qquad \subscr{\sigma}{max}(A) \leq 3\sqrt{N}/2,
\end{align*}
where $\subscr{\sigma}{min}$ ($\subscr{\sigma}{max}$) is the
smallest (largest) singular~value. \oprocend
\begin{proof}
  For notational convenience, we use $\subscr{\sigma}{min}$,
  $\subscr{\sigma}{max}$, and $\delta'$ to denote
  $\subscr{\sigma}{min}(A)$, $\subscr{\sigma}{max}(A)$, and
  $2\log{(1/\delta)}$, respectively. From \cite[Corollary
  5.35]{RV:10}, we have each of the following inequalities holds with probability at least $1-\delta$
\begin{align}\label{eq: svd inequality}
  \subscr{\sigma}{min} \geq \sqrt{N}-\sqrt{n} - \sqrt{\delta'}, \quad 
  \subscr{ \sigma}{max}\! \leq\! \sqrt{N} \!+ \sqrt{n} + \sqrt{\delta'}.
\end{align}
Assume that
$N\geq 8n +8\delta'$. Then,
\begin{align}\label{eq: N inequality}
 \sqrt{N}/2 \geq \sqrt{n} +\sqrt{\delta'},
\end{align}
where we have used the inequality $2(a^2+b^2) \geq (a+b)^2$. The proof
follows by substituting \eqref{eq: N inequality} into \eqref{eq: svd
  inequality}.
\end{proof}
\end{appxlem}
%
%
\renewcommand{\thesubsection}{B}
\subsection{Proof of Theorem \ref{thm: LQR gain}} \label{app: LQR}
Let $\subscr{u}{v}^* \in \real^{mT}$ and
$\subscr{x}{v}^* \in \real^{nT}$ be the optimal LQR trajectories of
\eqref{eq: system} from the initial state $x_0$. Then,
$\subscr{K}{LQR} = \subscr{u}{m}^* \subscr{x}{m}^{*^\dag}$
asymptotically as the control horizon $T$ grows. Further, from
\cite{FC-GB-FP:22,FC-GB-FP:22a}, the trajectories $\subscr{u}{v}^*$ and
$\subscr{x}{v}^*$ can be obtained using \eqref{eq: LQR trajectories}
when the state data is not corrupted by the process noise. Let
$\subscr{X}{clean}$ be such data, that is, the state trajectories of
\eqref{eq: system} with inputs $U$ and noise $w(t) = 0$ at all
times. Notice that in our setting $X$ is different from
$\subscr{X}{clean}$ since the process noise is nonzero when the data
is collected. Because of this deviation in the data, the vectors
$\subscr{u}{v}$ and $\subscr{x}{v}$ in \eqref{eq: LQR trajectories}
are a perturbed version of the optimal trajectories $\subscr{u}{v}^*$
and $\subscr{x}{v}^*$. Accordingly,
$K_\text{LQR}^\text{D} = \subscr{u}{m} \subscr{x}{m}^{\dag}$ is a
perturbed version of $\subscr{K}{LQR}$. To quantify the deviation
between $K_\text{LQR}^\text{D}$ and $\subscr{K}{LQR}$, we quantify (i)
the deviation in the data induced by the process noise, (ii) the
sensitivity of the map \eqref{eq: LQR trajectories} that generates LQR
trajectories, and (iii) how the induced errors propagate to compute
$K_\text{LQR}^\text{D}$.

\smallskip
\noindent
\emph{(i) Data deviation induced by the process noise.} Note that
\begin{align}\label{eq: data evolution2}
  X=\underbrace{\left[\begin{array}{c;{2pt/2pt}c}
                        O & F_u
                      \end{array}\right]}_{ F}
                            \underbrace{\left[\begin{array}{c}
                                                X_0\\
                                                \hdashline[2pt/2pt]
                                                U
                                              \end{array}\right]}_{\overline{U}}
  + F_wW,
\end{align}
where $W \in \mathbb{R}^{nT\times N}$ is a matrix that contains the
corresponding $N$ process noise realizations of horizon $T-1$, and
\begin{align*}
O= \left[\begin{smallarray}{c}
      I_n\\
      A\\
      \vdots \\
      A^T
    \end{smallarray}\right],~
 F_u=\left[\begin{smallarray}{ccc}
      0    & \cdots & 0\\
      B & \cdots & 0\\
      \vdots  & \ddots & \vdots \\
      A^{T-1}B & \cdots & B
    \end{smallarray}\right], ~
     F_w=\left[\begin{smallarray}{ccc}
      0    & \cdots & 0\\
      I_n & \cdots & 0\\
      \vdots  & \ddots & \vdots \\
      A^{T-1} & \cdots & I_n
    \end{smallarray}\right] .
\end{align*}
Note that $\subscr{X}{clean}=F\overbar{U}$. Let the data matrices in
\eqref{eq: data} and \eqref{eq: X0} be partitioned as
\begin{align}\label{eq: data partitioned}
  U =
  \begin{bmatrix}
    \subscr{U}{d} & \subscr{U}{n}
  \end{bmatrix},
                    X =
                    \begin{bmatrix}
                      \subscr{X}{d} & \subscr{X}{n}
                    \end{bmatrix},
                    X_0 =
                    \begin{bmatrix}
                      X_{0,\text{d}} & X_{0,\text{n}}
                    \end{bmatrix},                                   
\end{align}
where $\subscr{U}{d}$, $\subscr{X}{d}$, and $X_{0,\text{d}}$ contain
the first $N_\dd \geq mT+n$ columns of $U$, $X$, and $X_0$,
respectively, and let
$\overbar{U}=[\subscr{\overbar{U}}{d},\subscr{\overbar{U}}{n}]$ be
partitioned similarly. For notational convenience, we define
$Q_T= ( I_{T+1}\otimes Q_x )$ and
$R_T= \blkdiag(0_{n\times n }, I_{T} \otimes R_u )$. Noting that $\Ubar_\dd {\Ubar_\dd}^\dagger=I_{n+mT}$, we rewrite
$\subscr{u}{v}$ in \eqref{eq: LQR trajectories} as
\begin{align}\label{eq: LQR input}
      \subscr{u}{v}=H P^{-1/2}\left(
    \begin{bmatrix}
      I_n \!\!&\!\! 0_{n\times mT}
    \end{bmatrix}
            P^{-1/2}\right)^{\dagger} x_0 ,
  \end{align}
with
\begin{align}\label{eq: LQR P}
 P&\!=\!\left(\subscr{\widetilde{X}}{c}\subscr{\overbar{U}}{d}^{\dagger}\right)^{\T}\!Q_T \!\left(\subscr{\widetilde{X}}{c}\subscr{\overbar{U}}{d}^{\dagger}\right)\! +\! R_T,~\text{and}~\subscr{\widetilde{X}}{c} =\! X\overbar{U}^{\dagger}\subscr{\overbar{U}}{d}.
\end{align}
Further, let 
\begin{align}\label{eq: Xc}
\subscr{X}{c}=
  \subscr{X}{clean}\overbar{U}^{\dagger}\subscr{\overbar{U}}{d} \quad
  \text{and} \quad \Delta_X= \subscr{\widetilde{X}}{c}
  -\subscr{X}{c} .
\end{align}
Notice that if the process noise, $W$, is zero, then $\Delta_X=0$ and
$\subscr{\widetilde{X}}{c}=\subscr{X}{c}$ and, from \eqref{eq: LQR
  trajectories}, $\subscr{u}{v}=\subscr{u}{v}^*$ and
$\subscr{x}{v}=\subscr{x}{v}^*$. Thus, we use $\Delta_X$ as a proxy
for the deviation between $X$ and $\subscr{X}{clean}$, which is
induced by the process noise, $F_wW$. The next Lemma provides a
non-asymptotic upper bound to ${\|\Delta_X\|}_2$.
\smallskip
\begin{appxlem}{\bf \emph{(Non-asymptotic bound on ${\|\Delta_X\|}_2$)}}\label{lemma: Delta_X bound}
  Let $\Delta_X$ be as in \eqref{eq: Xc}, and let
  $\delta \in [0,1/3]$. Assume that
  $\left.N> \max{\{N_1,N_\dd\}}\right.$, with
  $N_1=2\left((n+m)T+n\right)\log{\left(1/\delta\right)}$ and
  $N_\dd\geq 8(mT+n) +16
  \log{\left(1/\delta\right)}$.
%
%
Then, with probability at least $1-3\delta$,
\begin{align}\label{eq: X_c error bound}
 {\|\Delta_X\|}_2 \leq d_1\sqrt{\frac{N_\dd \left((n+m)T+n\right)\log{\left(9/\delta\right)}}{N}},
\end{align}
where $d_1= 24 \|F_w\|_2 \|\overbar{Q}_w\|_2^{1/2}$ and
$\overbar{Q}_w= I_{T}\otimes Q_w$.  \oprocend
\begin{proof}
  Let $\overbar{U}=\Sigma_{\overbar{u}}^{1/2}Z$ and
  $\subscr{\overbar{U}}{d}=\Sigma_{\overbar{u}}^{1/2}\subscr{Z}{d}$,
  where
  $\Sigma_{\overline{u}}=
  \blkdiag{\left(\Sigma_0,I_{T}\otimes \Sigma_u\right)}$,
  $Z \in \mathbb{R}^{n+mT \times N}$ is a random matrix whose columns
  are independent copies of $\mathcal{N}\sim(0,I_{n+mT})$, and
  $\subscr{Z}{d}$ contains the first $N_\dd$
  columns of $Z$. From \eqref{eq: LQR P}, \eqref{eq: Xc},
\begin{align*}
 {\|\Delta_X\|}_2 &\!=\! {\|F_w \!W \overbar{U}^{\T}\! (\overbar{U}\overbar{U}^{\T})^{-1} \subscr{\overbar{U}}{d} \|}_2\!=\! {\|F_w\! W\! Z^{\T}\! (ZZ^{\T})^{-1} \subscr{Z}{d} \|}_2\\
 &\leq {\|F_w\|}_2 {\|W Z^{\T}\|}_2  {\|(ZZ^{\T})^{-1}\|}_2  {\|\subscr{Z}{d}\|}_2.
\end{align*}
The proof follows by using Lemma \ref{lemma: product of Gaussian matrices} to bound ${\|W Z^{\T}\|}_2$, Lemma \ref{lemma: svd of Gaussian matrix} to bound ${\|(ZZ^{\T})^{-1}\|}_2$ and ${\|\subscr{Z}{d}\|}_2$, and using the union bound to compute the probability.
\end{proof}
\end{appxlem}
\smallskip
\noindent
\emph{(ii) Sensitivity of map \eqref{eq: LQR trajectories}
  w.r.t. $\Delta_X$.} We focus our analysis on the map
$\left. \map{f}{\mathbb{R}^{n(T+1)N_\dd}\times
    \mathbb{R}^{(n+mT)N_\dd}}{\mathbb{R}^{n+mT}}\right.$
that generates $\subscr{u}{v}$ as in \eqref{eq: LQR input}. Then,
$\subscr{u}{v}^*=f(\Vect{(\subscr{X}{c})},\Vect{(\subscr{\Ubar}{d})})$. Since
$f$ is Fréchet-differentiable with respect to $\Vect{(\subscr{X}{c})}$
\cite{FC-GB-FP:22a,TK-DV:05}, we can write its first-order
Taylor-series~expansion~as
\begin{align}\label{eq: taylor series}
\begin{split}
f(\Vect{(\subscr{\widetilde{X}}{c})},\Vect{(\subscr{\Ubar}{d})})&\!=\!f(\Vect{(\subscr{X}{c})},\Vect{(\subscr{\Ubar}{d})})\\
 + \nabla f_X&\left(\Vect{(\subscr{X}{c})},\Vect{(\subscr{\Ubar}{d})}\right)  \!\Vect{(\Delta_X)},
\end{split}
\end{align}
where $\nabla f_X$ is the Jacobian matrix of
$f(\Vect{(\subscr{X}{c})},\Vect{(\subscr{\Ubar}{d})})$ with respect to
$\Vect{(\subscr{X}{c})}$.  We quantify the sensitivity of the map
\eqref{eq: LQR input} to the change in $\subscr{X}{c}$ by $\nabla f_X$
(large values of $\nabla f_X$ implies higher sensitivity).  Next, we
derive an upper bound on ${\|\nabla f_X\|}_2$, and upper bounds on
${\|\subscr{u}{v}-\subscr{u}{v}^*\|}_2$ and
${\|\subscr{x}{v}-\subscr{x}{v}^*\|}_2$ using the first-order
approximation in \eqref{eq: taylor series}.
\smallskip
\begin{appxlem}{\bf \emph{(Non-asymptotic bound on $\|\nabla f_X\|_2)$}}\label{lemma: jacobian bound}
  Let $\subscr{\overbar{U}}{d}$, $\subscr{X}{c}$, and
  $\nabla
  f_X\left(\Vect{(\subscr{X}{c})},\Vect{(\subscr{\overbar{U}}{d})}\right)$
  be as in \eqref{eq: Xc} and \eqref{eq: taylor series}. Also, let
  $\delta \in [0,1]$ and assume that
  $N_\dd\geq 8(n+mT)+16\log{(1/\delta)}$. Then, with
  probability at least $1-\delta$,
  \begin{align}\label{eq: jacobian bound}
    \left\|\nabla f_X\left(\Vect{(\subscr{X}{c})},\Vect{(\subscr{\overbar{U}}{d})}\right)\right\|_{2} \leq 4 d_2\sqrt{\frac{n(T+1)}{N_\dd}},
  \end{align}
  where $d_2> 0$ is independent of
  $N_\dd$. \oprocend
\end{appxlem}
\begin{proof}
  The proof can be adapted from the proof of \cite[Lemma
  IV.4]{FC-GB-FP:22a}, then using Lemma \ref{lemma: svd of Gaussian
    matrix} and $\|\cdot\|_2 \leq \|\cdot\|_{\mathrm{F}}$.
\end{proof}
\smallskip
\begin{appxthrm}{\bf \emph{(Non-asymptotic bound on the deviation of the LQR trajectories)}}\label{thrm: bound on LQR traj}
  Let $\subscr{u}{v}$ and $\subscr{x}{v}$ be as in \eqref{eq: LQR
    trajectories} and $\subscr{u}{v}^*$ and $\subscr{x}{v}^*$ be the
  optimal LQR trajectories of length $T$ of \eqref{eq: system} from
  the initial state $x_0$. Let $\delta \in [0,1/6]$ and assume that
  $N\geq \max{\left\{N_1,N_2,N_3\right\}}$, with
  $N_1=2\left((n+m)T+n\right)\log{\left(1/\delta\right)}$,
  $N_2=8(mT+n) +16 \log{\left(1/\delta\right)}$, and
  $N_3=((n+m)T+n) \log{\left(9/\delta\right)}$.
%
Then, with probability at least $1-4\delta$,
\begin{align}\label{eq: LQR input bound}
\begin{split}
\|\subscr{u}{v}-\subscr{u}{v}^*\|_2 &\leq d_3 \sqrt{\frac{\left((n+m)T+n\right)\log{(9/\delta)}}{N}}.
\end{split}
\end{align}
Further, with probability at least $1-6\delta$,
\begin{align}\label{eq: LQR trajectory bound}
\begin{split}
 \|\subscr{x}{v}-\subscr{x}{v}^*\|_2 \leq& d_4\sqrt{\frac{\left((n+m)T+n\right)\log{(9/\delta)}}{N}},
 \end{split}
\end{align}
with
\begin{align*}
d_3&=4d_1 d_2\sqrt{qn(T+1)},\\
d_4 &= {\|F\|}_2 d_3 +16{\|F_w\|}_2 {\|\Sigma_{\overbar{u}}^{-1/2}\|}_2 {\|\overbar{Q}_w\|}_2^{1/2}\left({\|\subscr{\overbar{u}}{v}^*\|}_2+d_3\right),
\end{align*}
where $d_1$, $F$, and $F_w$ are as in \eqref{eq: X_c error bound}
and~\eqref{eq: data evolution2},~respectively, $d_2>0$ is
independent of $N$, $q=\Rank{(\Delta_X)} \leq n(T+1)$,
$\subscr{\overbar{u}}{v}^*= [x_0^{\T}, {\subscr{u}{v}^*}^{\T}]^{\T}$,
and $\overbar{Q}_w$ and $\Sigma_{\overline{u}}$ are as in Lemma
\ref{lemma: Delta_X bound}. \oprocend
\begin{proof}
Inequality \eqref{eq: LQR input bound} follows from \eqref{eq: taylor series} by using Lemma \ref{lemma: Delta_X bound}, Lemma \ref{lemma: jacobian bound}, and ${\|\Vect{(\Delta_X)}\|}_2={\|\Delta_X\|}_{\mathrm{F}}\leq\sqrt{q}{\|\Delta_X\|}_{2}$, with $q=\Rank{(\Delta_X)}$. Next, we derive \eqref{eq: LQR trajectory bound}. For notational convenience, we use $\Delta_u$ and $\Delta_x$ to denote $ \subscr{u}{v}-\subscr{u}{v}^*$ and $ \subscr{x}{v}-\subscr{x}{v}^*$, respectively. From \eqref{eq: LQR trajectories}, we can write
\begin{align}\label{eq: Delta_x proof}
 {\|\Delta_x\|}_2&=\left\|\subscr{\widetilde{X}}{c}\subscr{\overbar{U}}{d}^{\dagger} 
\begin{bmatrix}
 x_0\\ \subscr{u}{v}
\end{bmatrix}
-\subscr{X}{c}\subscr{\overbar{U}}{d}^{\dagger} 
\begin{bmatrix}
 x_0\\ \subscr{u}{v}^*
\end{bmatrix}\right\|_2\\
&=\left\|\subscr{X}{c}\subscr{\overbar{U}}{d}^{\dagger}
\begin{bmatrix}
 0 \\ \Delta_u
\end{bmatrix}
+\Delta_X \subscr{\overbar{U}}{d}^{\dagger}
\begin{bmatrix}
 x_0 \\ \subscr{u}{v}^*
\end{bmatrix}
+\Delta_X \subscr{\overbar{U}}{d}^{\dagger}
\begin{bmatrix}
 0 \\ \Delta_u
\end{bmatrix}\right\|_2 \nonumber\\
\leq
\| \subscr{X}{c}&\subscr{\overbar{U}}{d}^{\dagger}\|_2 \|\Delta_u\|_2
+\|\Delta_X \subscr{\overbar{U}}{d}^{\dagger}\|_2 \|\subscr{\overbar{u}}{v}^*\|_2
+\|\Delta_X \subscr{\overbar{U}}{d}^{\dagger}\|_2 \|\Delta_u\|_2.\nonumber
\end{align}
Note that $\subscr{\overbar{U}}{d} \subscr{\overbar{U}}{d}^{\dagger}=I_{n+mT}$. Then we have 
\begin{align*}
{ \|\subscr{X}{c}\subscr{\overbar{U}}{d}^{\dagger}\|}_2&=\|\subscr{X}{clean}\Ubar^{\dagger}\|_2=\|F\|_2,\\
{ \|\Delta_X\subscr{\Ubar}{d}^{\dagger}\|}_2&\!=\!\!\|(X\!-\!\subscr{X}{clean})\Ubar^{\dagger}\|_2 =\|F_w W \Ubar^{\dagger}\|_2\\
&\!\leq\!\|F_w\|_2\|W \Ubar^{\T}\|_2\|(\Ubar \Ubar^{\T})^{-1}\|_2,
\end{align*}
Inequality \eqref{eq: LQR trajectory bound} follows from \eqref{eq: Delta_x proof} by using \eqref{eq: LQR input bound}, Lemma \ref{lemma: product of Gaussian matrices}, and Lemma \ref{lemma: svd of Gaussian matrix} to bound $\|\Delta_u\|_2$, $\|W \Ubar^{\T}\|_2$, and $\|(\Ubar \Ubar^{\T})^{-1}\|_2$, respectively, and noting that for $N\geq N_3$ we have $\frac{\delta'}{N}\leq \sqrt{\frac{\delta'}{N}}$, with $\delta'=\left((n+m)T+n\right)\log{(9/\delta)}$.
%
The probabilities follow from the union bound.
\end{proof}
\end{appxthrm}

\smallskip \emph{(iii) Error between $\subscr{K}{LQR}$ and
  $K_\text{LQR}^\text{D}$.} We are now ready to conclude the proof of
Theorem \ref{thm: LQR gain}. Notice that
\begin{align}
  \subscr{u}{m}^* = \subscr{K}{LQR} 
  \subscr{x}{m}^*, \text{ and }
  \subscr{u}{m}^* + \delta_u = \subscr{K}{LQR}^\text{D} \left(
  \subscr{x}{m}^* + \delta_x \right),
\end{align}
where $\delta_u=\subscr{u}{m}-\subscr{u}{m}^*$ and
$\delta_x=\subscr{x}{m}-\subscr{x}{m}^*$. Note that $\subscr{u}{m}$
and $\subscr{x}{m}$ are the matrices obtained by reorganizing the
inputs and states in the vectors $\subscr{u}{v}$ and $\subscr{x}{v}$
in chronological order. For notational convenience, we use $K$ and
$K^{\text{D}}$ to denote $\subscr{K}{LQR}$ and
$ \subscr{K}{LQR}^\text{D}$. Let $\Delta_K=K-K^{\text{D}}$. In what
follows, subscript $i$ denotes the $i$-th row, with
$i\in \{1,\cdots , m\}$. Using~\cite[Theorem 5.1]{PAW:73} and assuming that $\subscr{x}{m}^*$ is of full row rank,\footnote{This condition is typically satisfied for generic choices of the initial~state.}
\begin{align}\label{eq: error K_LQR proof}
{\|\Delta_{K,i}\|}_2\leq d_5\left(\epsilon {\|K_i\|}_2 {\|\subscr{x}{m}^*\|}_2 + {\|\delta_{u,i}\|}_2  + \epsilon \alpha{\|r\|}_2 \right),
\end{align}
where 
\begin{align*}
 d_5=\frac{\alpha}{1-\alpha \epsilon {\|\subscr{x}{m}^*\|}_2}, \quad \epsilon = \frac{{\|\delta_x\|}_2}{{\|\subscr{x}{m}^*\|}_2}, \quad r= \subscr{u}{m,$i$}^*  -K_i \subscr{x}{m}^*,
 \end{align*}
 and $\alpha=\|\subscr{x}{m}^*\|_2\|{\subscr{x}{m}^*}^{\dagger}\|_2$ is the spectral condition number of $\subscr{x}{m}^*$. From \cite[Theorem 3.2]{FC-GB-FP:22}, we have $\|r\|_2\!\leq\! d_6\rho^T$, where $d_6\!>\!0$ and $\rho\!<\!1$, which are independent of $N$. Since ${\|\subscr{x}{m}^*\|}_2\!=\!\subscr{\sigma}{max}(\subscr{x}{m}^*)$, ${\|(\subscr{x}{m}^*)^{\dagger}\|}_2\!=\!1/\subscr{\sigma}{min}(\subscr{x}{m}^*)$. Then, we can write $d_5$~as 
 \begin{align*}
 d_5=\frac{1}{\subscr{\sigma}{min}(\subscr{x}{m}^*)\left(1-  \kappa (\subscr{x}{m}^*) \right)}, \quad\text{with} \quad \kappa (\subscr{x}{m}^*)=\frac{\subscr{\sigma}{max}(\delta_x)}{\subscr{\sigma}{min}(\subscr{x}{m}^*)},
\end{align*}
For sufficiently large $N$ such that  $\subscr{\sigma}{max}(\Delta X)<\subscr{\sigma}{min}(X^*)$, we have $\kappa (\subscr{x}{m}^*)<1$ and $\epsilon \alpha <1$. Then, we can write \eqref{eq: error K_LQR proof} as
 \begin{align*}
{\|\Delta_{K,i}\|}_2&\leq d_5\Big({\|K_i\|}_2 {\|\delta_x\|}_2 + {\|\delta_{u,i}\|}_2  + d_6\rho^T\Big),\\
&\overset{\text{(a)}}{\leq} d_5 \Big({\|K\|}_2 {\|\Delta_x\|}_2 + {\|\Delta_u\|}_2  + d_6\rho^T\Big),
\end{align*}
where in step (a), we have used $\|\delta_x\|_2 \leq \|\delta_x\|_{\mathrm{F}}=\|\Vect{(\delta_x})\|_{2}=\|\Delta_x\|_{2}$, and $\|\delta_{u,i}\|_2=\|\delta_{u,i}\|_{\mathrm{F}}\leq \|\delta_{u}\|_{\mathrm{F}}=\|\Vect{(\delta_u)}\|_{2}=\|\Delta_u\|_{2}$, where $\Delta_x$ and $\Delta_u$ are as in Theorem \ref{thrm: bound on LQR traj}. Noting that $\|\Delta K\|_{\mathrm{F}}=\sqrt{\Tr{\Delta K (\Delta K)^{\T}}}=\sqrt{\sum_{i=1}^m \|\Delta K_i\|_2^2}$ and using the bounds in Theorem \ref{thrm: bound on LQR traj}, we have with probability at least $1-6\delta$
  \begin{align*}
    {\| \Delta_K \|}_2 \leq \frac{1}{\subscr{\sigma}{min}(\subscr{x}{m}^*)\left(1-\kappa(\subscr{x}{m}^*)\right)} \left(\frac{c_1}{\sqrt{N}} + c_2 \rho^T \right), 
  \end{align*}
where,
\begin{align}\label{eq: c1 and c2}
\begin{split}
 c_1&\!=\!\left(d_3\!+\!\|\subscr{K}{LQR}\|_2 d_4\right)\! \sqrt{m\left((n\!+\!m)T\!+\!n\right)\log{(9/\delta)}},\\
 c_2&=d_6\sqrt{m},
\end{split}
\end{align}
and $d_3$ and $d_4$ are as in Theorem \ref{thrm: bound on LQR traj}. Finally, the probability follows from the union bound. This concludes the proof.
\renewcommand{\thesubsection}{C}
\subsection{Proof of Theorem \ref{thm: KF}} \label{app: KF}
The Kalman filter computes the estimate $x_{\kf}(t)$ given
$\{u(0),\dots, u(t-1), y(0),\dots, y(t)\}$ that minimizes the cost
\begin{align}\label{eq: MMSE}
   \sum_{t=0}^T
  \expect{(x(t)-x_{\kf}(t))^{T}(x(t) - x_{\kf}(t))} ,
\end{align}
which is then used to generate LQG inputs. Equivalently,
$x_{\kf}(t)$ can be obtained with the following linear estimator,
\begin{align*}
x_{\kf}(t)\!\!=\!\!&
\underbrace{\begin{bmatrix}
L^u_{t,0}  &\!\!\!\! \cdots\!\!\!\!&\!\! L^u_{t,t-1}
\end{bmatrix}}_{L^u_t}
\!\underbrace{\begin{bmatrix}
 u(0)\\
 \vdots\\
 u(t-1)
\end{bmatrix}}_{u_0^{t-1}}
\!+\!
\underbrace{\begin{bmatrix}
L^y_{t,0}  & \!\!\!\!\cdots\!\!\!\! &\!\! L^y_{t,t}
\end{bmatrix}}_{L^y_t}
\!\underbrace{\begin{bmatrix}
 y(0)\\
 \vdots\\
 y(t)
\end{bmatrix}}_{y_0^{t}}\!,\\
=&
\underbrace{\begin{bmatrix}
 L^u_t & L^y_t
\end{bmatrix}}_{L_t^{\text{KF}}}
\underbrace{\begin{bmatrix}
 u_0^{t-1}\\y_0^{t}
\end{bmatrix}}_{z_t},
\end{align*}
where $L_t^{\text{KF}} \in \mathbb{R}^{n\times mt+p(t+1)}$, with
$L^u_t \in \mathbb{R}^{n\times mt}$ and
$L^y_t \in \mathbb{R}^{n\times p(t+1)}$, is the estimator gain that
minimizes \eqref{eq: MMSE}. Let $e(t) = x(t)- x_{\kf}(t)$ and
$\Sigma_{e,t} \in \mathbb{R}^{n\times n}\succeq 0$ denote the
estimation error and the estimation error covariance matrix,
respectively. For an optimal linear estimator, $L_t^{\text{KF}}$, we have
$e(t)\sim \mathcal{N}(0,\Sigma_{e,t})$, and we can write the state
$x(t)$ as
\begin{align*}
 x(t)=L_t^{\text{KF}} z_t + e(t).
\end{align*}
Let 
\begin{align}\label{eq: state and estimation error t}
  x_t = [x^1(t),\dots ,x^N(0)], \quad e_t =  [e^1(t),\dots , e^N(0)],
\end{align}
where $x^i(t)$ and $e^i(t)$ denote the state and the state estimation
error incurred by $L_t^{\text{KF}}$ at time $t$ for the $i$-th trajectory of the
data \eqref{eq: data}, respectively. Further, let
$Z_t = [U_{t-1}^{\T}, Y_t^{\T}]^{\T}$, where $U_t$ and $Y_t$ are the
submatrices of $U$ and $Y$ in \eqref{eq: data} obtained by selecting
the inputs and outputs up to some $t$. Then,
\begin{align}\label{eq: estimator data t}
x_t=L_t^{\text{KF}} Z_t + e_t.
\end{align}
To estimate the optimal filter $L_t$ from the data \eqref{eq: data},
we consider the following least squares problem
\begin{align}\label{eq: least squares Lt}
 L_t^{\text{D}}=\argmin{L_t}{\|x_t-L_t Z_t\|_{\mathrm{F}}^2}.
\end{align}
Problem \eqref{eq: least squares Lt} admits a unique solution since $Z_t$ is full-row rank, which is given by \eqref{eq: DD KF}. Next, we bound $\|L_t^{\text{D}}-L_t^{\text{KF}}\|_2$.
\smallskip
\begin{appxthrm}{\bf \emph{(Non-asymptotic bound on
      $\|L_t^{\D}-L_t^{\text{KF}}\|_2$)}}\label{lemma: Lt_hat sample
    complexity}
  Let $L_t^{\text{KF}}$ and $L_t^{\text{D}}$ be as in \eqref{eq:
    estimator data t} and \eqref{eq: DD KF}, respectively, and let
  $\delta \in [0,1/2]$. Assume that
  $N\geq \max{\left\{N_1,N_2\right\}}$, with
  $N_1=2\left((n+m)T+n\right)\log{\left(1/\delta\right)}$ and
  $N_2=8(mT+n) +16 \log{\left(1/\delta\right)}$.
  %
  Then, with probability at least $1-2\delta$,
  \begin{align}\label{eq: F_hat error bound}
    {\|L_t^{\text{D}}-L_t^{\text{KF}}\|}_2 \leq d_7\sqrt{\frac{\left((m+p)t+n+p\right)\log{\left(9/\delta\right)}}{N}},
  \end{align}
  with 
  \begin{align*}
    d_7\triangleq 16 {\|\Sigma_Z^{-1/2}\|}_2 {\|\Sigma_{e,t}\|}_2^{1/2} , 
  \end{align*}
  where $\Sigma_Z \in \mathbb{R}^{(m+p)t+p\times (m+p)t+p}\succ 0$
  comprises the noise statistics, $\Sigma_0$, and $\Sigma_u$ in Assumption
  \ref{assump: data input}, and $\Sigma_{e,t}$ is the optimal
  estimation error covariance matrix at time $t$. \oprocend
\end{appxthrm}
\begin{proof}
  Let $Z=\Sigma_{Z}^{1/2}G$ and where $\Sigma_{Z}\succ 0$ is as in the
  theorem statement, and $G \in \mathbb{R}^{(m+p)t+p \times N}$ is a
  random matrix whose columns are independent random vectors
  distributed as $\mathcal{N}\!\sim\!(0,I_{(m+p)t+p})$. From
  \eqref{eq: DD KF} and \eqref{eq: estimator data t},
\begin{align*}
{\|L_t^{\D} -L_t^{\text{KF}}\|}_2&={\|e_t Z^{\dagger}\|}_2={\|e_t G^{\T} (G G^{\T})^{-1} \Sigma_Z^{-1/2}\|}_2\\
&\leq {\|\Sigma_Z^{-1/2}\|}_2 {\|e_t G^{\T} \|}_2 {\|(G G^{\T})^{-1}\|}_2.
\end{align*}
The proof follows by using Lemma \ref{lemma: product of Gaussian
  matrices} to bound ${\|e_t G^{\T}\|}_2$, and Lemma \ref{lemma: svd
  of Gaussian matrix} to bound ${\|(GG^{\T})^{-1}\|}_2$. Finally, The
probability follows from the union bound.
\end{proof}
To conclude the proof of Theorem \ref{thm: KF}, we have
\begin{align*}
   {\left\|\subscr{x}{KF}(t)- L_t^{\D} 
        \begin{bmatrix*}
      u_0^{t-1} \\ y_0^t
    \end{bmatrix*}\right\|}_2
    &={\left\|L_t^{\text{KF}}
    \begin{bmatrix*}
      u_0^{t-1} \\ y_0^t
    \end{bmatrix*}- L_t^{\D} 
        \begin{bmatrix*}
      u_0^{t-1} \\ y_0^t
    \end{bmatrix*}\right\|}_2\\
    &\leq {\|L_t^{\text{KF}}-L_t^{\D}\|}_2 {\left\|
    \begin{bmatrix*}
      u_0^{t-1} \\ y_0^t
    \end{bmatrix*}\right\|}_2,
  \end{align*}
  where $u_0^{t}$ and $y_0^t$ are the vectors of inputs and outputs of
  \eqref{eq: system}, respectively, from time $0$ up to time
  $t$. Using Theorem \ref{lemma: Lt_hat sample complexity},
\begin{align}\label{eq: state estimate bound}
   {\left\|\subscr{x}{KF}(t)- L_t^{\D} 
        \begin{bmatrix*}
      u_0^{t-1} \\ y_0^t
    \end{bmatrix*}\right\|}_2
    \leq \frac{c_3}{\sqrt{N}}{\left\|
    \begin{bmatrix*}
      u_0^{t-1} \\ y_0^t
    \end{bmatrix*}\right\|}_2,
  \end{align}
where $c_3=d_7\sqrt{\left((m+p)t+n+p\right)\log{\left(9/\delta\right)}}$, and $d_7$ is as~in Theorem \ref{lemma: Lt_hat sample complexity}. The above inequality holds with probability at least $1-2\delta$, which follows from Theorem \ref{lemma: Lt_hat sample complexity} for $\delta~\in~[0,1/2]$. This concludes the proof of Theorem \ref{thm: KF}.

\renewcommand{\thesubsection}{D}
\subsection{Proof of Theorem \ref{thm: LQG}}\label{app: LQG}
%
Consider the closed-loop trajectories in \eqref{eq: cl data}, and let
$U^n_{\dlqg}$ and $Y_{\dlqg}^n$ be the submatrices of $U_\dlqg$ and
$Y_\dlqg$ in \eqref{eq: cl data} obtained by selecting only the inputs
from time $T-n$ up to time $T-1$ and the outputs from time $T-n+1$
up~to~time~$T$, respectively. We can write the data-based and the
model-based LQG inputs at time $T$ for the trajectories in \eqref{eq:
  cl data} as
\begin{align*}
\underbrace{\begin{bmatrix}
 u^1_{\dlqg}(T) & \cdots &  u^M_{\dlqg}(T)
\end{bmatrix}}_{U_{\dlqg}(T)}&=K_{\lqr}^{\D} L_T^{\D} 
\begin{bmatrix}
 U_\dlqg\\
 Y_\dlqg
\end{bmatrix},\\
\underbrace{\begin{bmatrix}
 u_{\lqg}^1(T)&\! \cdots \!\!& \!\! u_{\lqg}^M(T)
\end{bmatrix}}_{U_{\lqg}(T)}&=K_{\lqr}L_T^{\kf} 
\underbrace{\begin{bmatrix}
 U_\dlqg\\
 Y_\dlqg
\end{bmatrix}}_{Z}.
\end{align*}
For notational convenience, let $\Delta{K_{\lqr}}=K_{\lqr}^{\D}-K_\lqr$, $\Delta L=L_T^{\D}-L_T^\kf$, and $\Delta U= U_\dlqg(T) -U_\lqg(T)$. Then,
\begin{align}\label{eq: delta_U}
\begin{split}
 \Delta U&=K_{\lqr}^{\D}L_T^{\D}Z-K_{\lqr}L_T^{\kf}Z\\
&=K_{\lqr}\Delta L Z + \Delta K_{\lqr} L_{T}^{\kf}Z+\Delta K_\lqr \Delta L Z.
\end{split}
\end{align}
For sufficiently large $T$, we use \eqref{eq: static LQG} to write
\begin{align*}
U_{\dlqg}(T)=
\subscr{K}{LQG}^{\D}
\begin{bmatrix}
 U^n_{\dlqg}\\
Y^n_{\dlqg}
\end{bmatrix},
U_{\lqg}(T)=\subscr{K}{LQG}
\underbrace{\begin{bmatrix}
U^n_{\dlqg}\\
Y^n_{\dlqg}
\end{bmatrix}}_{Z_n}.
\end{align*}
Then, $K^{\D}_{\lqg}\!\!=\!\!U_\dlqg(T)Z_n^{\dagger}$ and
$K_{\lqg}\!=\!U_\lqg(T)Z_n^{\dagger}$. For notational convenience,
let $\Delta K_\lqg=K_\lqg^{\D} - K_\lqg$, and let $\|\cdot\|$ denote
$\|\cdot\|_2$. Then, using \eqref{eq: delta_U}, we can write
\begin{align}\label{eq: LQG bound proof}
\begin{split}
\| \Delta &K_\lqg \|_2=\|(U_\dlqg(T)-U_\lqg(T)) Z_n^{\dagger}\|=\|\Delta U Z_n^{\dagger}\|\\
\leq&\|K_{\lqr}\| \|\Delta L\| \| Z Z_n^\dagger\| + \|\Delta
K_{\lqr}\| \|L_{T}^{\kf}\| \|Z Z_n^\dagger\|\\
&+ \|\Delta K_\lqr\| \|\Delta L\| \|Z Z_n^\dagger\|.
\end{split}
\end{align}
Let $\delta \in [0,1/8]$ and assume that
$N\geq \max{\left\{N_1,N_2,N_3\right\}}$, where $N_1$, $N_2$, and
$N_3$ are as in Theorem \ref{thrm: bound on LQR traj}. Then,
inequality \eqref{eq: LQG bound} follows by using Theorem \ref{thm:
  LQR gain} and Theorem \ref{lemma: Lt_hat sample complexity} to bound
$\|\Delta K_{\lqr}\|$ and $\|\Delta L\|$ in \eqref{eq: LQG bound
  proof}, respectively, with probability at least $1-8\delta$ and with
\begin{align}\label{eq: c5 c6 c7}
\begin{split}
 c_5&\!=\!\frac{c_1\|L_t^{\kf}\| + c_1c_3}{\subscr{\sigma}{min}(\subscr{x}{m}^*)\left(1-\kappa(\subscr{x}{m}^*)\right)}+c_3\|K_\lqr\|,\\
 c_6&\!=\!\frac{c_2 c_3 }{\subscr{\sigma}{min}(\subscr{x}{m}^*)\!\left(1\!\!-\!\kappa(\subscr{x}{m}^*)\right)},~  c_7\!=\!\frac{c_2 \|L_t^{\kf}\| }{\subscr{\sigma}{min}(\subscr{x}{m}^*)\!\left(1\!-\!\kappa(\subscr{x}{m}^*)\right)},
 \end{split}
\end{align}
where $c_1$, $c_2$, $\subscr{x}{m}^*$, and $\kappa(\subscr{x}{m}^*)$ are as in Theorem \ref{thm: LQR gain}, and $c_3$ is as in Theorem \ref{thm: KF}. Finally, the probability follows using the union bound. This concludes the proof of Theorem \ref{thm: LQG}.
%
%
\bibliographystyle{unsrt}
\bibliography{alias,Main,FP,New}

\end{document}